\documentclass[a4paper,11pt,leqno]{amsart}

\usepackage{graphics}
\usepackage{thmtools}
\usepackage{wasysym}

\usepackage[T1]{fontenc}    

\usepackage{amsthm}
\usepackage{amsbsy,amsmath,amssymb,amscd,amsfonts}
\usepackage[pagebackref=true]{hyperref}
\usepackage{url}            
\usepackage{booktabs}       
\usepackage{nicefrac}       
\usepackage{microtype}      

\usepackage{graphicx,float,latexsym,color}
\usepackage[font={scriptsize,it}]{caption}
\usepackage{subcaption}

\usepackage{makecell}

\usepackage[dvipsnames]{xcolor}

\newtheorem{theorem}{Theorem}
\newtheorem{A}{A}

\newtheorem{proposition}{Proposition}

\newtheorem{lemma}{Lemma}
\theoremstyle{remark}

\theoremstyle{definition}
\newtheorem{definition}{Definition}

\hypersetup{
    pdftoolbar=true,        
    pdfmenubar=true,        
    pdffitwindow=false,     
    pdfstartview={FitH},    
    colorlinks=true,       
    linkcolor=OliveGreen,          
    citecolor=blue,        
    filecolor=black,      
    urlcolor=red           
}

\usepackage{lineno}

\usepackage{comment}
\usepackage{microtype}

\title{Poncelet and the arquimedean twins}
\author{Liliana Gabriela  Gheorghe}
\begin{document}
\maketitle

\textbf{Abstract.}
\small{We give a sharp  construction for twins in arbelos, based on polar reciprocity. In the process, new circles displaying arquimedean afinities came into scene.
}

  \begin{figure}
\centering
\includegraphics[trim=50 100 50 120,clip,width=0.9\textwidth]{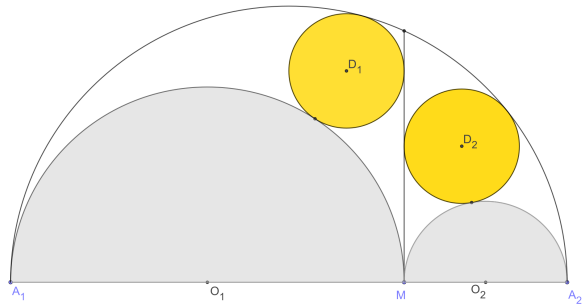}
\caption{Twins in arbelos are solutions of two distinct Apollonius' problems: but why are they arquimedean?}
\label{fig:isca1}
\end{figure}

\section{Introduction}

Arbelos   is the greek for 'shoemaker's knife', the shape  bounded by three pairwise tangent semicircles with diameters lying on the same line. As a geometric object, it was first studied by Arquimedes
 in his Book of Lemmas, hence  dates back more than 2200 years ago. If we draw the common (internal)  tangent to the arbelos circles  then the
two circles that
 tangent this linee,  the arbelos outer circle and   one of the arbelos' circles are called "the twins"; see figure \ref{fig:isca}.
Arquimedes had already spotted them 
and proved that their radius is half the harmonic mean of the two arbelos i-circles.
As a tribute, 
circles in  arbelos, congruent with the twins are called arquimedean;  chasing  arquimedean circles  in arbelos 
was is a  constant theme ever-since.

Perhaps
the most humble arquimedean circle is
those whose diameter is the parallel through $M$ at the bases of the (rectangular) trapeze of basis $R_1,R_2$ and altitude $R_1+R_2$;
see figure \ref{fig:isca2}; in fact the proof of this fact  uses elementary proprieties in trapezius. Longer, yet 
straightforward computation   confirms that
the tangents from $O_1$ and $O_2$
to this circle meet precisely at the center of the arbelos i-circle. This circle was spotted in [B] (somehow backwardly that presented here), but we presume Arquimedes already knew about it.
\subsection*{Main results}
Rather than chasing arquimedean circles,
here we are interested in explain  why the twins are identical and how to draw them. We  foreseen the (classic) twins as solutions of two degenerated Apollonius' problem, and we find their centers as intersection of special conics. 

\vspace{0.3cm}

{\bf{Keywords:}} {arbelos, Apollonius's problem, circle inversion, poles, polar duals.}

{\bf 2020 Mathematics Subject Classification: 51A05, 51A30, 51M15.}

To  perform their intersection,  we use polar reciprocity, the method tailored by Poncelet ([P]) in order to proof his Porism.
The proof itself led  to a sharp geometric construction for their centers, which is our main result; see theorem \ref{thm:construction}.
The method also explains why  the radii of the twins is the same.
We then attach another  pair of circles, which has twice the radius of the initial one, and study  a new tern  of related circles.
One of the circles of the tern is  arquimedean, and was first spotted by Scotch in [S];  the other two are new and  verify an arquimedean-type relation which we prove in theorem \ref{thm:cousins}.

\subsection*{Related work}
The literature on arbelos is so rich and the references are so abundant, that we   limit ourselves to  a small sample that show that interest for this problem is still vivid.
In [B] a famous construction for the i-circle in arbelos, much simpler than the one
from Archimedes’ proof, due to the existence of a third arquimedean circle;
[W] also provides a simple construction of the i-circle in (classic) arbelos; in
[D-L] the authors
study inversions mapping that switch  two given circles and apply it to arbelos;
[O] studies a generalization of arbelos;
 [O-W] study twin circles in skewed arbelos. [We] is a handy source that collect many of these facts and related bibliography.

 \begin{figure}
\centering
\includegraphics[trim=20 0 20 0,clip,width=0.9\textwidth]{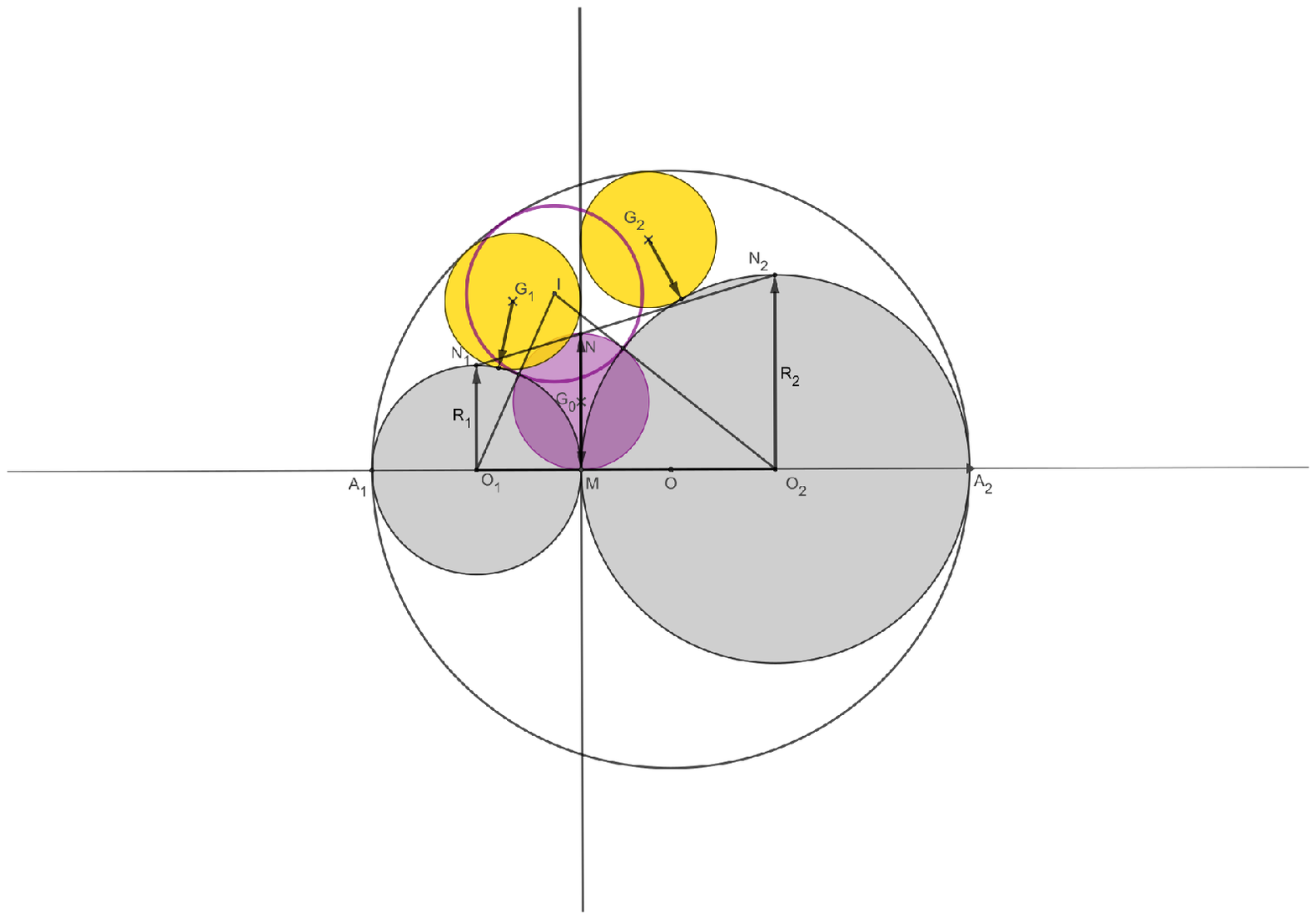}
\caption{An arquimedean circle (bordeaux) canonically attached to an arbelos; two other  circles  of same radius (solid yellow) can be drawn  tangent to  the arbelos' circles (grey) and to their common tangent; but in this case, why do they also the arbelos' external circle?}
\label{fig:isca2}
\end{figure}

\subsection*{Notations}
We shall note by $(O)$ a circle centered in $O$. By reflection in $(O)$ we mean the symmetry (or inversion) with respect to circle $(O).$ We freely use  "dual curve", "polar dual" or "reciprocal curve" as synonymous.

\section{Twins' centers and the Apolonius' problem}

Twins in arbelos are the two circles that  tangent the arbelos'  nested circles and their common (internal) tangent line; thus,
each twin is a solution of an Apollonius problem; hence, as in [G],  their centers can be obtained by intersecting some special conics.

Let a circle and a tangent line to it be given, as in figure \ref{fig:locus}.
\begin{lemma} The locus of the centers of the circles that tangent both the circle and the line is a parabola focused at  circle's  center  and whose vertex is at the tangency point of the circle and the line.
  \label{lema:parabola}
\end{lemma}

Let $(O)$ and $(O_1)$ be two internally tangent circles, as in figure \ref{fig:locus}.
\begin{lemma} The locus of the centers of the
 circles that tangent two nested (internally) tangent  circles is an ellipse focused at the centers of the two nested circles  and passing through their common tangency point.
\label{lemma:elipse}
\end{lemma}
Both lemma \ref{lema:parabola} and   \ref{lemma:elipse} have elementary proofs that we omit.
When we specialize to   arbelos, we get a sharper result.
Refer to figure  \ref{fig:locus}.
\begin{proposition}
The center of each twin
is the intersection between an ellipse and a parabola.

Each ellipse has one vertex at the common
tangency point of each internal arbelos circle with the external one, (points $A_1$ and $A_2$ respectively), has  one focus in $O$ and the other focus in  $O_1$ and $O_2,$ respectively.
The vertex of the parabolas  is the common tangency point of the internal arbelos'  circles $(O_1)$, $(O_2)$  (point $M$) and their focus is in $O_1$ and $O_2,$ respectively.
\label{proposition:twin-generic}
\end{proposition}

At this point, any drawing software is able to perform the  intersection of these two conics. Nevertheless,
 we are "not done", since the task is  to perform a geometric construction: a construction with a straight-line and a compass only. In a geometric construction, one cannot  intercept  "continuous curves",  other then circles and lines.


In fact,  while two arbitrary conics cannot be (geometrically) intersected, the conics that aroused here,  have a special feature: a common focus.
And here is where polar reciprocity comes into scene. The reader not acquainted with this topic, may see Appendix.

 \begin{lemma}
  The intersection  of two
 conics that have a common focus are the poles of the common tangents of their duals, w.r. to an inversion circle centered at their (common) focus.
 \label{lemma:int-conics}
 \end{lemma}

For more details on poles, polars and polar reciprocity, see [A], [S], [GSO].
  \begin{figure}
\centering
\includegraphics[trim=500 180 600 410,clip,width=0.9\textwidth]{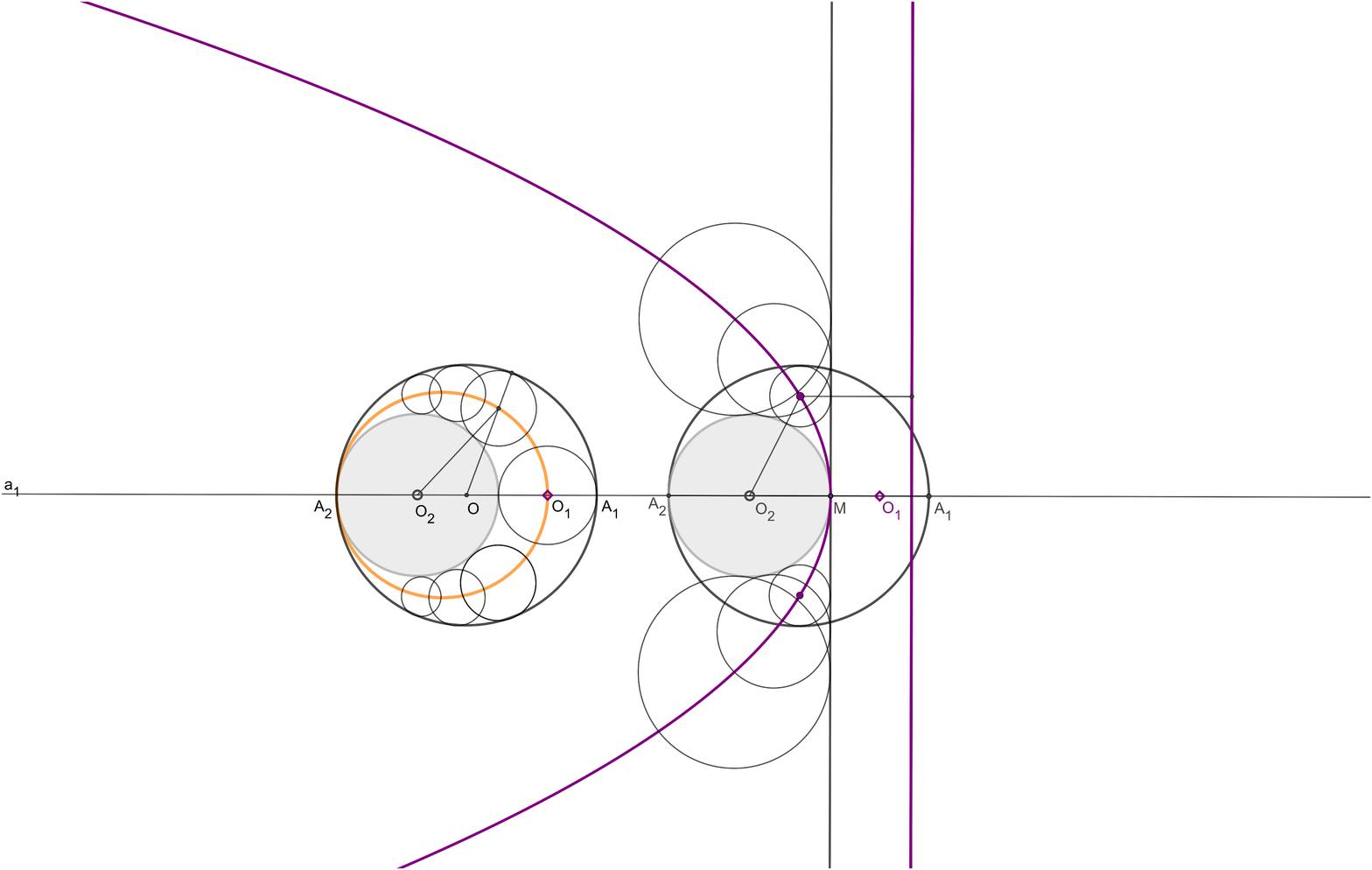}
\caption{I) {\bf (left)} The locus of the centers of the circles that tangents two internally tangent circles is an ellipse (orange) which has the foci at the centers of the two circles and one vertice at their common tangency point. II) {\bf (right)} The locus of the centers of the circles that tangents externally one circle (solid grey) and a  line (which tangent the circle) is a parabola (purple), which has the focus into the center of the  circle and one vertice at the common tangency point of the circle and the line.
}
\label{fig:locus}
\end{figure}

  \begin{figure}
\centering
\includegraphics[trim=400 240 400 200,clip,width=0.9\textwidth]{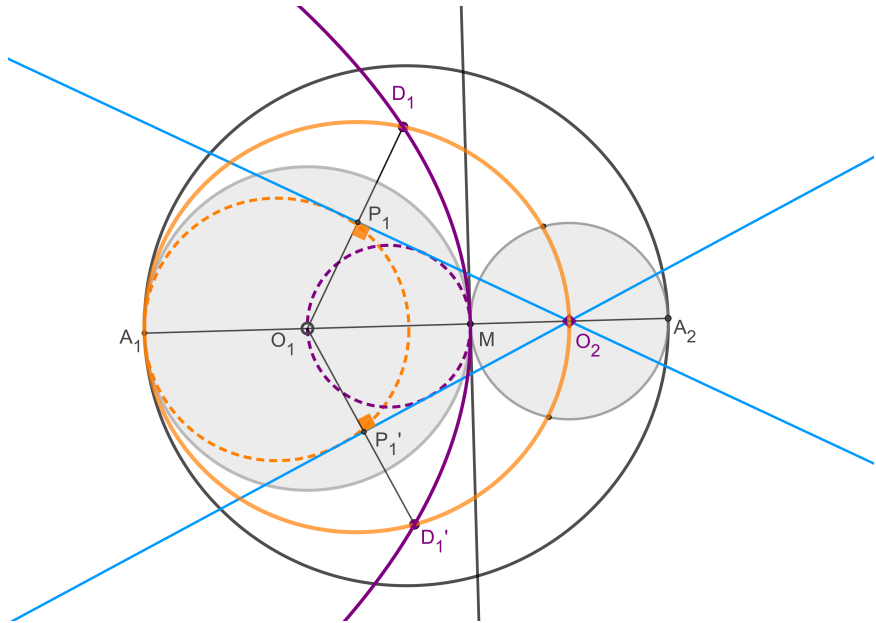}
\caption{I) The polar dual w.r. to $(O_1)$ of the ellipse focused in
$O_1$ and $O$ and passing through $A_1$
is a circle (dotted orange)
tangent at $A_1$ to $(O)$. Its diameter is $[A_1 O_2']$ where $O_2'$ is the reflection of $O_2$ in $(O).$
II) The polar dual  w.r. to $(O_2)$  of the parabola focused in $O_1$ and vertex $M$ is the circle of diameter $[O_2 M].$
III) The similitude center of these two circles is $O_2,$ the center of the second arbelos' circle.
IV) The two (real) intersection points  of the parabola and the ellipse are the poles
(w.r. to $(O_2)$)
of their common tangents :
shown is point $D_1,$
 the center of one of the arbelos' twins
}
\label{fig:rec}
\end{figure}

 \begin{figure}
\centering
\includegraphics[trim=300 100 300 50,clip,width=0.9\textwidth]
{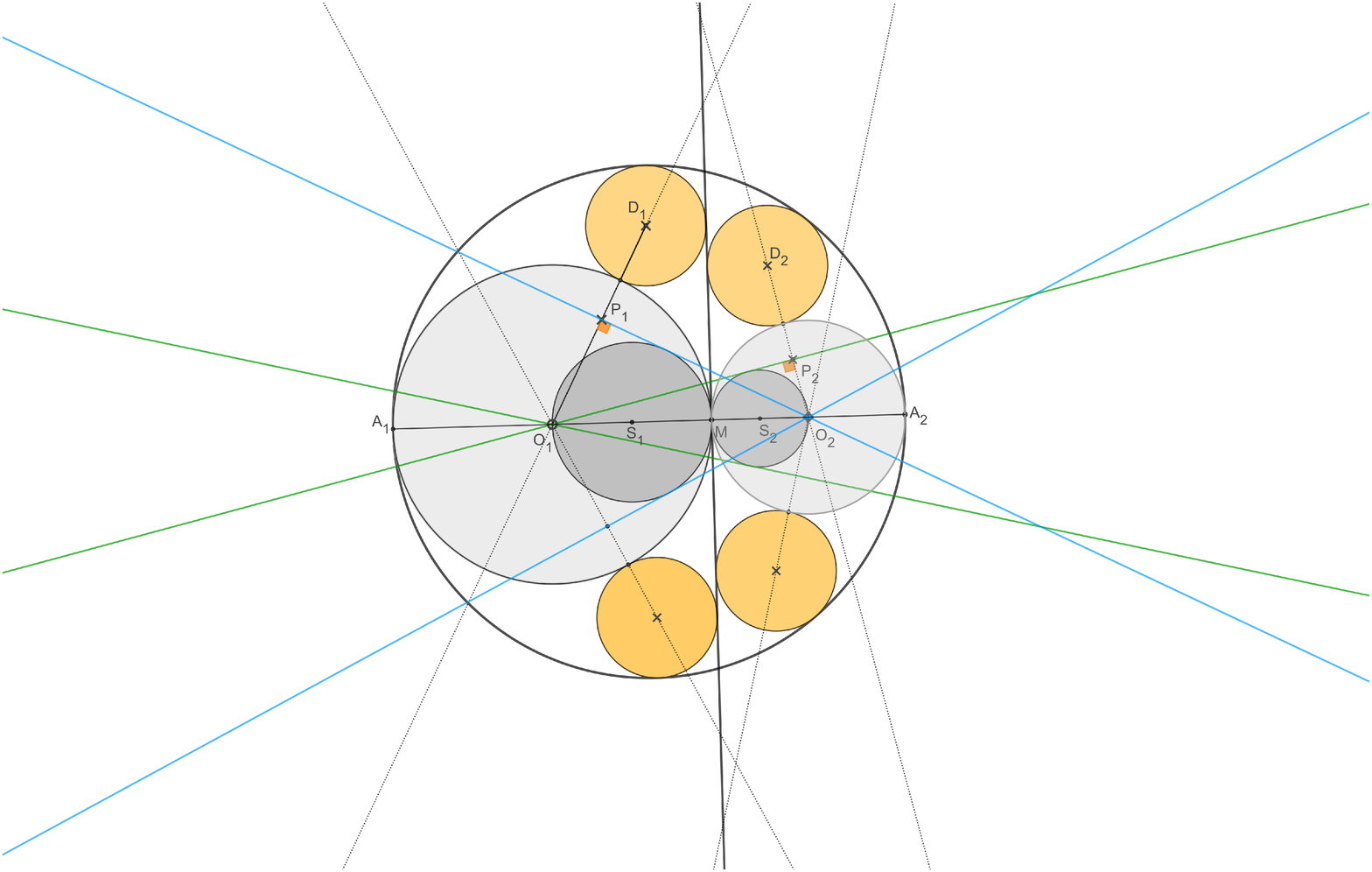}
\caption{Twins in arbelos: a solution via siblings,  tangents, and poles.}
\label{fig:isca}
\end{figure}

Now we may specify who these duals are.
Refer to figure \ref{fig:rec} and choose as the inversion circle $(O_1)$.
 \begin{lemma}
The polar duals (w.r. to $(O_1)$)   of  the ellipses $\mathcal {E}_1$ and the parabola
$\mathcal{P}_1$
in figure \ref{fig:rec}, are two circles that:
\begin{enumerate}
\item
are  tangent to $(O_1)$ at $A_1$ and  $M$ respectively;
\item the diameter of the parabola's dual is $[O_1 M]=R_1$
and the diameter of the ellipse's dual is $[A_1 O_2']= \frac{2R_1^2+R_1 R_2}{R_1+R_2} $
where $O_2'$ is the reflection of $O_2$ in $(O_1);$
\item the similitude centers of these two circles is $O_2.$
\end{enumerate}
 \label{lemma:homotetia}
 \end{lemma}
 \begin{proof}
 The proof uses known facts on the polar of a conic, that we collect in  Appendix. Both the parabola and the ellipse in figure \ref{fig:rec} have one common focus in $O_1;$ therefore, their duals w.r. to $(O_1)$ are circles.
 The dual of a parabola w.r. to an inversion circle centered at its focus is a circle, whose  diameter is
 $[O_1 M]$, where $M$ is
 the reflection of the parabola's vertice and  $O_1$ is the  center of inversion. Since the vertice $M$ is located on the inversion circle, its is invariant by  reflection; thus, the  dual of the parabola is a circle with diameter $[O_1M]=R_1.$

 The ellipse focused in $O_1$ and $O$, and passing through $A_1$ has its second  vertice at $O_2.$
 Therefore, its polar dual w.r. to  $(O_1)$
 is a circle whose diameter is $[A_1 O_2']$, where $O_2'$ is the reflection of $O_2$ in $(O_1)$; the vertice $A_1$  is invariant, since it is  a point of the inversion circle. Thus, the diameter  of the dual is
 \begin{equation}
 A_1 O_2'=A_1 O_1+O_1O_2'=R_1+\frac{R_1^2  }{R_1+R_2}=
 \frac{2R_1^2+R_1 R_2}{R_1+R_2}.
 \label{eq:homot1}
\end{equation}
Let $o_1$ and $S_1$ be the centers of these dual  circles; their radius are, respectively
\begin{equation}
r_1=\frac{A_1M}{2}=
 \frac{2R_1^2+R_1 R_2}{2(R_1+R_2)} \;\;\mbox{ and}\;\;
 r_m=\frac{R_1}{2}.
 \label{eq:raiosrec1}
 \end{equation}

 In order to prove that the similitude center is $O_2,$ we prove that
 \begin{equation}
\frac{o_1 O_2}{S_1 O_2}=\frac{r_1}{r_m}.
\label{eq:homot2}
\end{equation}

In fact, \[S_1O_2=\frac{R_1}{2}+{R_2}=\frac{R_1+2R_2}{2}\] and
\[o_1 O_2=A_1O_2-A_1o_1=2R_1-\frac{R_1(2R_1+ R_2)}{2(R_1+R_2)} =\frac{(2R_1+R_2)(R_1+2R_2)}{2(R_1+R_2)};          \] hence
\[\frac{o_1 O_2}{S_1 O_2}=
\frac{(2R_1+R_2)}{(R_1+R_2)};\]
\[
\frac{r_1}{r_m}= \frac{R_1(2R_1+R_2)}{2(R_1+R_2)}\cdot \frac{2}{R_1}=\frac{(2R_1+R_2)}{(R_1+R_2)},    \]
hence equation \ref{eq:homot2} is verified.

 \end{proof}

We may attach to any arbelos  a pair of tangent circles, "the siblings:" these are two circles, mutually tangent  at the common tangency points of the arbelos circles, and of half radius each.

The results proved above justifies the following new and sharp construction of the centers of the twins;
refer to figure \ref{fig:isca}.
\begin{theorem}
The centers of the twins are
 the poles of the tangents drawn from the center of a arbelos' circle, to the opposite sibling.
 \label{thm:construction}
\end{theorem}

  \section{Twins' radii}

At this point, we draw the twins' centers  using polar reciprocity. This does not guarantees  that  the circles are arquimedean.
Fortunately,
the method is proper for computation purposes, as well.
Refer to figure \ref{fig:isca}.
\begin{lemma}
 Let $(O_1)$ and $(O_2)$ the (internal) arbelos circles and let $(S_1)$ and $(S_2)$,  be their siblings: two circles  whose  diameters are  $[O_1M]$ and $[O_2M]$.
 Let $O_1T_2$ the tangent from $O_1$ to circle $(S_2);$
 let $D_2$ be its pole w.r. to $(O_2)$; construct similarly $D_1.$
  Then \[O_2D_2-R_2=O_1D_1-R_1.\]
 \label{lemma:radius}
\end{lemma}

\begin{proof} Let $r_1,r_2$ be half of the radius $R_1,R_2$. By construction,
\[\triangle{O_1 O_2 P_2}\sim \triangle{O_1 S_2 T_2},\] hence
\[\frac{r_2}{O_2P_2}=\frac{2r_1+r_2}{2r_1+2r_2},
\;\;\; \mbox{or}\;\;\; O_2P_2=   \frac{2r_2(r_1+r_2)}{2r_1+r_2}. \] Since
$D_2,$ the pole of the line $O_1 T_2$
is the reflection of $P_2,$ the projection of $O_2$  to  the line $O_1 T_2$, with respect to  $(O_2).$
\[O_2 P_2\cdot  O_2 D_2= 4r_2 ^2,\] hence
\[O_2 D_2=\frac{4r_2^2}{O_2P_2}=
\frac{4r_2^2(2r_1+r_2)}{2r_2(r_1+r_2)}=\frac{2r_2(2r_1+r_2)  }{ r_1+r_2 } .\]
This gives
   \[ O_2 D_2-R_2=O_2 D_2-2r_2=\frac{2r_2(2r_1+r_2)  }{ r_1+r_2 }-2r_2= 2r_2\bigg[\frac{(2r_1+r_2)}{r_1+r_2}-1\bigg],\] hence

 \begin{equation}
r=O_2D_2-R_2=\frac{2r_1\cdot r_2}{r_1+r_2}.
\label{eq:raio-gemeo}
\end{equation}
Formula \ref{eq:raio-gemeo} is symmetric with respect to $r_1$ and $r_2;$
an identical computation, obtained by switching the indices $1$ and $2,$ ends the proof.

\end{proof} Lemma \ref{lemma:int-conics} and Lemma \ref{lemma:radius}, prove the  result on the twins;
refer to figure \ref{fig:isca}.
\begin{theorem}
The circles centered in $D_1,D_2$ and of radius $r$ are arquimedean twins in arbelos: they are tangent to both $(O_1)$ $(O_2)$
 and $(O)$, and also are tangents to the line $l$ and their  (common) radius $r$ is
\[\frac{1}{r}=
\Big[ \frac{1}{R_1}+  \frac{1}{R_2}\Big] , \] where $R_1,R_2$ denotes the radii of the arbelos' circles.
\label{thm:radius-arbelos}
\end{theorem}
\begin{proof}
The fact that the points $D_1$ and $D_2$ are the centers of the twins is guaranteed by Lemma \ref{lemma:int-conics}. Thus, the circle centered in $D_2$ and whose radius  is   $D_2O_2-R_2$ tangents the line $l,$ as well, hence these are the twins. The relation \ref{eq:raio-gemeo} ends the proof.
\end{proof}

  \begin{figure}
\centering
\includegraphics[trim=290 150 240 150,clip,width=0.9\textwidth]{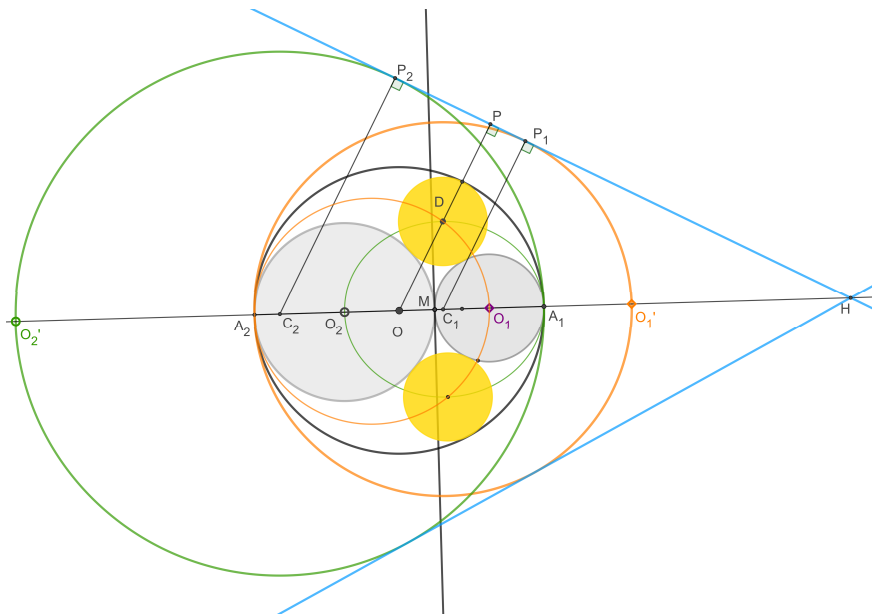}
\caption{
The center of the arbelos' i-circle as intersection of two ellipses:
one focused in $O$ and $O_1$ and passing through $O_2$, (green) and the other focused in $O$ and $O_2$ and passing through $O_1.$
Their intersection is the pole w.r. to $(O)$ of the common tangent (blue) to their reciprocal circles (orange and green circles).
}
\label{fig:circtg}
\end{figure}

\subsection*{Arbelos i-circle} For the sake of completeness, and because is  effortless,
we show how  may draw  the  arbelos' i-circle
$\mathcal{I}$, the circle that  tangents the three arbelos' circles.
\begin{theorem}
Refer to figure \ref{fig:circtg}.
\begin{enumerate}
    \item
The center of $\mathcal{I}$  is obtainable as the  intersection of two ellipses:
one focused in $O$, and $O_1$, and axis $[O_2A_1]$  and the other focused in
$O$ and $O_2$ and axis $[O_1A_2]$.
\item The intersection points of these ellipses are  the poles   of the common tangents  to their reciprocal circles, w.r. to $(O)$.
\item The radius of the arbelos i-circle is
$$R=\frac{R_1R_2(R_1+R_2)}{R_1^2+R_1R_2+R_2^2}$$.

\end{enumerate}

\end{theorem}

\begin{proof}
The two ellipses have a common focus in $O$ and tangent the inversion circle $(O)$ at their vertices $A_1$ and $A_2,$ respectively. Hence
their duals are two circles of diameters
$[A_1 O_2']  $ and $[A_2 O_1']$ respectively, where
$O_1'$ is the reflection of $O_1$ in $(O)$ and
$O_2'$ is the reflection of $O_2$ in $(O)$.

Let $H$ be the  similitude center of the reciprocal circles.

Then \[ \triangle{H C_1 P_1}\sim\triangle{H O P}\sim\triangle{H C_2 P_2}.\] Hence
\[OP=\frac{(R_1+R_2)(R_1^2+R_1R_2+R_2^2)        }{ R_1^2+R_2^2} .\]
Since $D$ is the pole of $HP,$
\[OD\cdot OP=(R_1+R_2)^2       .\]
The radius of the i-circle is \[ R=(R_1+R_2)-OD=(R_1+R_2)-\frac{(R_1+R_2)^2}{OP}       \]
which gives
\[R=\frac{R_1R_2(R_1+R_2)}{R_1^2+R_1R_2+R_2^2}.\]

\end{proof}



 \begin{figure}
\centering
\includegraphics[trim=80 100 0 40,clip,width=1.0\textwidth]{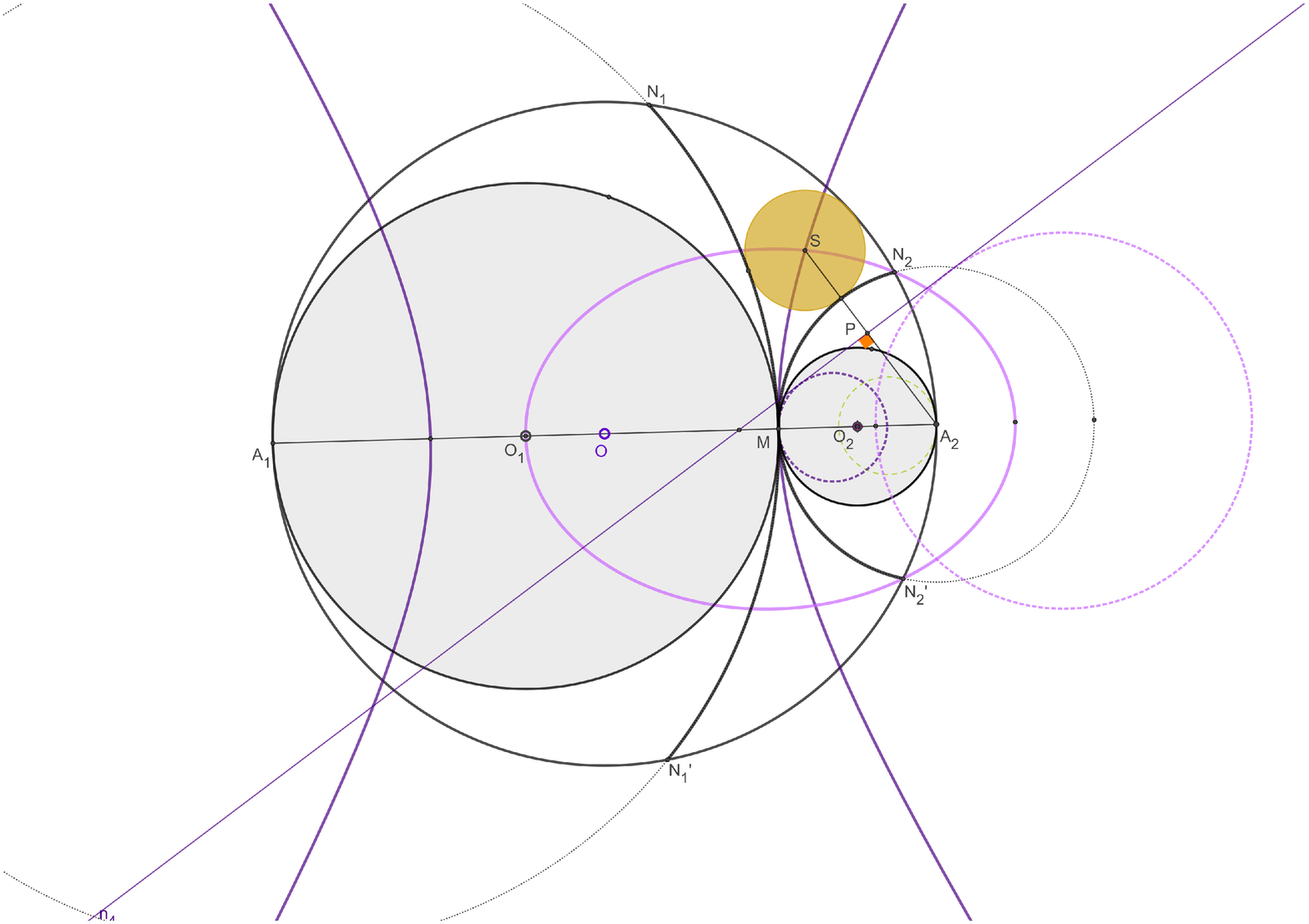}
\caption{ The cousin i-circle  (golden) tangents internally $(O)$ and externally $(A_1)$ and $(A_2)$. Its center obtains as intersection of an ellipse (violet) and a hyperbola, focused in $A_2,O$ and $A_1,A_2,$ respectively. $S$ is the pole of the common tangent at their reciprocal circles, w.r. to   $(A_2);$ its radius is the same as those of the classic twins:
$ \frac{1}{s}= \big[\frac{1}{R_1}+\frac{1}{R_2}\big]. $
}
\label{fig:icirc-cousin}
\end{figure}

\section{ Arquimedean circles in doubling arbelos}
We now associate to a classic arbelos, two other circles, passing through
the common tangency points $M$ and   centered at each of  the end-point of arbelo's diameter; we call it a doubling arbelos; see figure \ref{fig:icirc-cousin}  and \ref{fig:final}.

By   i-circle in a doubling arbelos, we mean the circle that tangents externally the two new-added circles and internally the arbelos diametral circle. Such circle  was first spoted by  Scotch, in [S], who proved that  it is  arquimedean; above, we indicate another proof for this fact, and also a construction for the center of this circle, both based on polar reciprocity.

\begin{theorem}
The  i-circle in a doubling arbelos is arquimedean:
$ \frac{1}{s}= \big[\frac{1}{R_1}+\frac{1}{R_2}\big]. $
\label{thm:cousins}
\end{theorem}
\begin{proof} Refer to figure \ref{fig:icirc-cousin}.
 We foresee the i-circle associated to a doubling arbelos' as a
 solution of an Apolonius' problem; therefore, its center $S,$ can be obtained as an intersection between an ellipses focused in $O$ and $A_2,$ passing through $O_1$ and a hyperbola focused in $A_1$ and $A_2,$ and passing through $M.$ Since these two conics have a common focus in $A_2,$ a polar dual w.r. to $(A_2)$ maps them into a pair of circles, whose diameters are the the reflection of their (ellipses and hyperbola's) vertices.
 $S,$
 the center of the  i-circle in a doubling arbelos is the pole of the common tangent to their reciprocal circles. Straightforward computations, similar to those in the proof of
 theorem \ref{thm:radius-arbelos} provide an give  proof  to the Scotch's result.

 These method embeds the geometric construction of $S$, as a pole of the common tangents to these reciprocal circles.

\end{proof}
Now we consider two new circles associated to a doubling arbelos:
the twin-cousins are  the circles that tangents externally the two arbelo's circles of  diameters $[A_1 A_2]$ and $[A_1 M]$, as well as internally the arbelos external circle and one of the (new) circles centered at $A_1$ and of radii $A_1 M$ or $A_2$ and of radii $A_2M;$ see figure \ref{fig:final} (bottom side).

The twin-cousins (solid blue circles) are not congruent; nevertheless  they verify an arquimedean-type metric relation.

\begin{theorem}
Let $s_1,s_2$ be the radius of the twin-cousins; then
\[\frac{1}{s_1}+\frac{1}{s_2}=3\big[ \frac{1}{R_1}+\frac{1}{R_2}     \big]\]
\label{thm:cousin-arq}
\end{theorem}
 \begin{figure}
\centering
\includegraphics[trim=240 70 90 30,clip,width=1.0\textwidth]{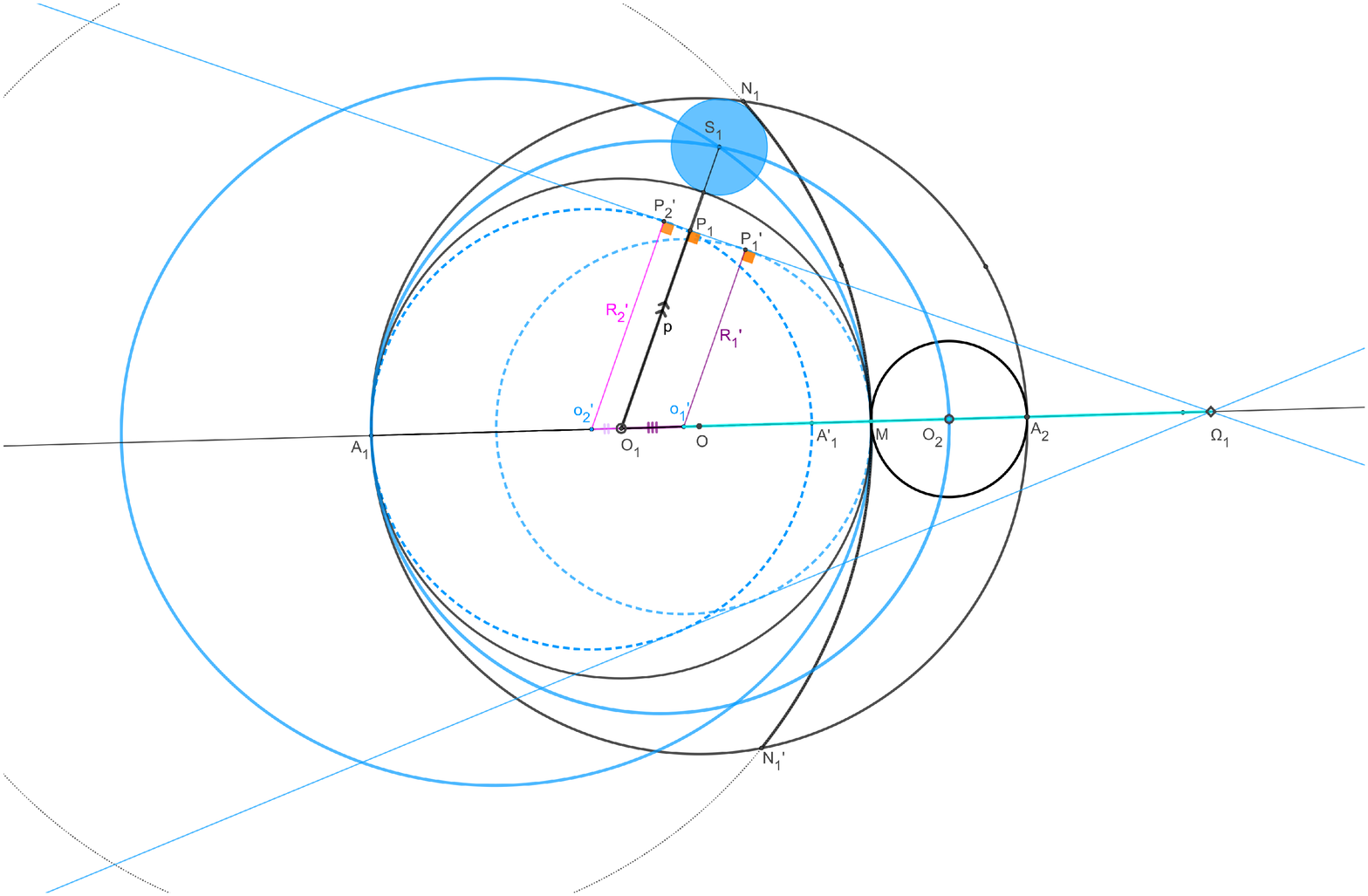}
\caption{$S_1,$ the center of a twin-cousin, obtained as intersection of two ellipses, as well as a pole of the common tangent at their dual circles (dotted, centered at $o_1'$
and $o_2'$, respectively.
}
\label{fig:lemacousin}
\end{figure}

\begin{figure}
\centering
\includegraphics[trim=200 100 150 50,clip,width=1.0\textwidth]{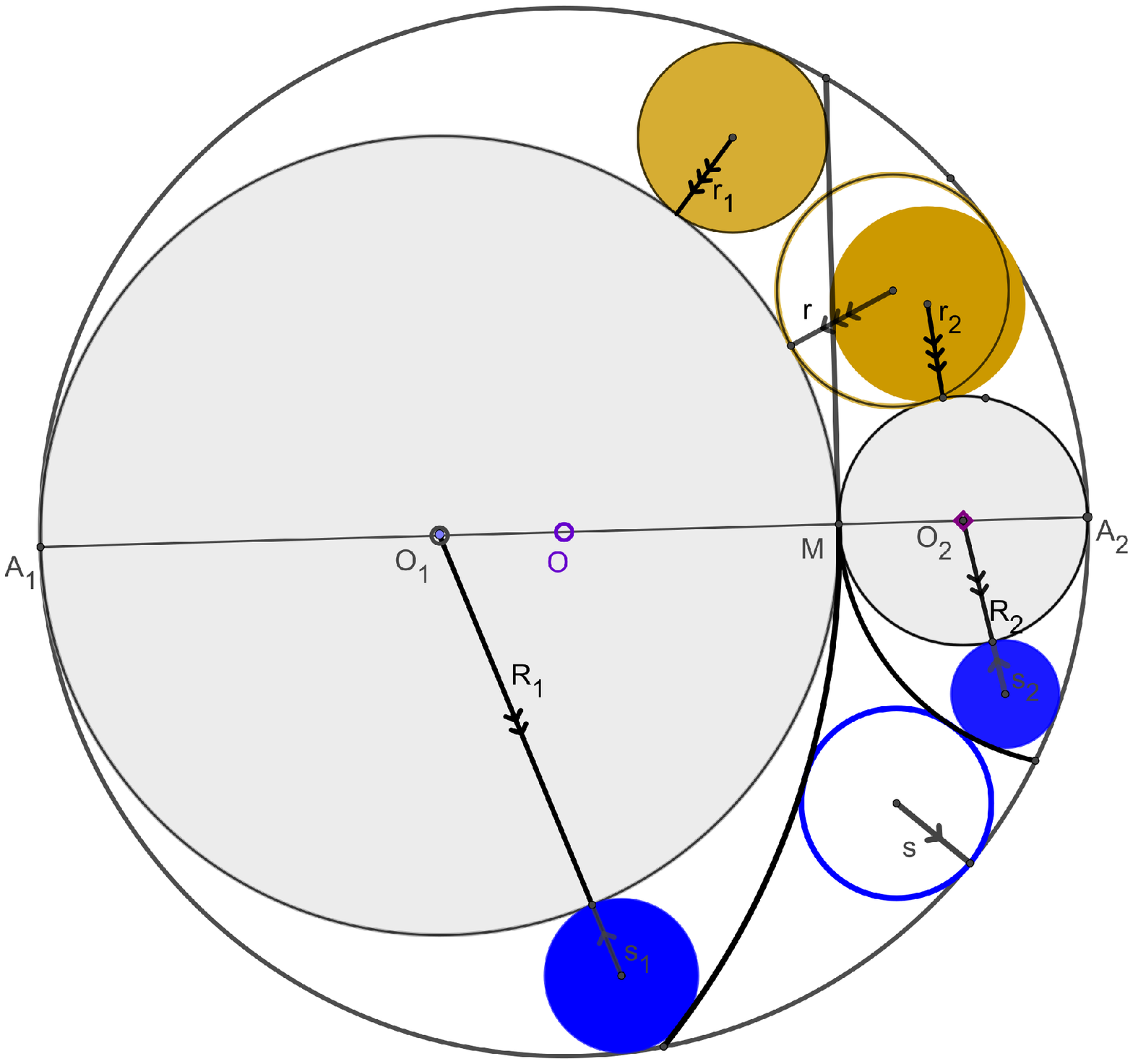}
\caption{
Metric coincidences in arbelos:
I) The (classic)  twins (solid golden) and the cousin i-circle (blue) are congruent and  arquimedean: $s=r_1=r_2$ and  $\frac{1}{s}=\frac{1}{R_1}+\frac{1}{R_2}$. II)
The two cousin-circles (solid blue)
verifies
$\frac{3}{s}=\big[\frac{1}{s_1}+\frac{1}{s_2}\big]$ hence $ \frac{1}{s_1}+\frac{1}{s_2} = 3\big[\frac{1}{R_1}+\frac{1}{R_2} \big]$
}
\label{fig:final}
\end{figure}

\begin{proof}
Refer to figure \ref{fig:lemacousin}. The center of  twin cousin $S_1$,  obtains as intersection of two ellipses:
$\mathcal{E}_1$ the ellipses focused in $O$ and $O_1$, and passing through $A_1,$
and another ellipses,
$\mathcal{E}'$, focused in $O_1$ and $A_1$, and passing through $M.$
Perform a dual transform w.r. to $(O_1)$;
since $O_1$ is the common focus of the ellipses $\mathcal{E}_1$ and $\mathcal{E}',$
their duals are two circles. The dual of $\mathcal{E}_1$ is the circle $\mathcal{C}'_2$ of diameter
$[A_1 O_2'],$ where $O_2'$ is the reflection of $O_2$ in $(O_1)$; denote by $o_2'$ and $R_2'$ be its center and radius.
Similarly,
 the dual of $\mathcal{E}'$ is the circle $\mathcal{C}'_2$ of diameter
$[B_1' M],$ where $B_1'$ is the reflection of the ellipse's vertice $B_1$  in $(O_1)$; denote by let $o_1'$ and $R_1'$ be its center and radius.
The common tangent, $p$ of these dual circles intercept the line of centers in $\Omega_1,$
their similitude center.
Let $P_1,P_1',P_2'$ denote respectively the projections of $O_1,o_1',o_2'$ on the tangent $p.$

Since
$\mathcal{E}'$ the ellipses focused in $O_1$ and $A_1$, and passing through $M,$ and since
$O_1 B_1=2\cdot R_1$,
\[O_1 B_1\cdot O_1 B_1'=R_1^2,\;\;\;O_1 B_1'=\frac{R_1}{2};   \]
therefore,
\begin{equation}
R_1'=\frac{3R_1}{4},\;\;\; O_1 O_1'=\frac{R_1}{4}.
\end{equation}
Now we find the segment $O_1 O_2'$.
The diameter $A_1A_1'$ of the dual circle of the ellipses $\mathcal{E}'$ is
\[A_1A_1'=O_1A_1+O_1A_1'=R_1+\frac{R_1^2}{R_1+R_2}=\frac{2R_1^2+R_1R_2}{R_1+R_2};    \]
hence
\begin{equation}
    R_1'=O_1A_1'=\frac{R_1^2}{R_1+R_2}
    \label{eq:$R_1'$}
    \end{equation}
Finally,
\[O_1O_2'=\frac{A_1A_1'}{2}-O_1A_1'=
\frac{2R_1^2+R_1R_2}{2(R_1+R_2)}-
\frac{R_1^2}{(R_1+R_2)}=\frac{R_2^2}{2(R_1+R_2)}.\]
Thus,
\[O_1 O_2'=O_1 A_1-O_2'A_1=R_1-\frac{A_1 A_1'}{2}= R_1-\frac{2R_1^2+R_1 R_2}{2(R_1+R_2)} . \]
In other words, \[O_1 O_1'=O_1 M-O_1'M=R_1-R_1'=
\frac{R_1 R_2}{2(R_1+R_2)}.
\]
and \begin{equation}
O_1' O_2'=O_1O_2'+O_1O_1'= \frac{R_1 R_2}{2(R_1+R_2)}+\frac{R_1}{4}=\frac{R_1^2+3R_1R_2}{4(R_1+R_2)}
\end{equation}
Now let $x=\Omega_1 O_1'$ and $p=O_1 P_1;$
by hypothesis,
\[\triangle{\Omega O_1'P_1'}\sim \triangle{\Omega O_1P_1}\sim
\triangle{\Omega O_2'P_2'}.\]
Using the relations above, we obtain
\[x=\frac{3R_1(R_1+3R_2)}{4(R_1-R_2)},\;\;\;
p=\frac{R_1(R_1+2R_2)}{(R_1+3R_2)}
.\]
Hence
\[O_1'O_2'=\frac{R_1^2+3R_1R_2}{4(R_1+R_2)},\;\;
R_1'=\frac{3R_1}{4},\;\;
R_2'=\frac{2R_1^2+R_1R_2}{2(R_1+R_2)}.
\]
With these ingredients in place, a straightforward computations led to
the radius of the  cousin-twin centered at $S_1$
\begin{equation}
    s_1=\frac{R_1 R_2}{R_1+2R_2}.
 \label{eq:cousin1}
\end{equation}
If we interchange the indices $1$ and $2$, we obtain
\begin{equation}
    s_2=\frac{R_1 R_2}{R_2+2R_1}.
    \label{eq:cousin2}
    \end{equation}
Thus
\[\frac{1}{s_1}+\frac{1}{s_2}=
3\big[\frac{1}{R_1}+\frac{1}{R_2}\big].\]

\end{proof}



\section*{Appendix. Brief recall on polar reciprocity}
We include here a
brief recall on polar reciprocity.

Let us fix
$C(\Omega,R)$  a circle centered in  $\Omega$ and of radius  $R,$
which we shall call inversion circle.

\begin{definition}
If $p_0$
is a line that does not pass through $\Omega,$
its pole  is the inverse of the projection of the center $\Omega,$ on the line $p_0.$
\end{definition}

\begin{definition}
If $P_0$
is a point $(P_0\neq \Omega),$
the polar of  $P_0$
is the perpendicular line on
$\Omega P_0,$  that pass through $P_1,$ the inverse of $P_0.$
\end{definition}
\begin{definition}
The polar
dual (or a reciprocal curve) of a regular curve (w.r. to an inversion circle) defines as
the curve whose points are the poles of the tangents of the original curve.
\end{definition}

When we perform the dual of a  circle, w.r. to an inversion
circle, we obtain conics.

\begin{A}(see e.g. [S], art. 306 and 309) The dual
	of a circle
	$\gamma=C(O,r),$ w.r. to an inversion circle
	${C(\Omega,R)},$  is a conic, ${\Gamma};$
	if   $d$ denotes the distance between the centres of the reciprocated and inversion circles, $d=\Omega O,$ then:
	
	i) $\Gamma$ is an ellipse, if  $r<d;$
	
	ii) $\Gamma$ is a parabola,  if $r=d;$
	
	iii) $\Gamma$ is a hyperbola, if  $r>d.$

	Moreover, (one of) the
	the focus of the dual conic
	$\Gamma$  is  precisely $\Omega,$ the center of the inversion circle;
	its directrix   is the polar of $O,$ the center of the reciprocated circle
	and the eccentricity  is
	$e=\frac{r}{d}.$

\end{A}

This theorem has a very useful counterpart.

\begin{A}
	The dual of a conic $\Gamma,$ w.r. to an inversion circle centered
	into its focus,  is a circle, $\gamma.$
	The symmetric of the vertices of the conic $\Gamma,$
	are a pair of diametrically opposite points of the dual  circle, $\gamma.$
The pole of the directrix of $\Gamma,$
	is the center of the circle $\gamma.$
	\label{thm:dualcurve}
\end{A}
All the geometrical elements of the dual conic,
can be drawn with straight-line and compass, since all the steps involves drawing the symmetric of a point and the pole of a line.

A final  useful fact regards intersections of (regular) curves.
\begin{A}
The intersection of two curves are the poles of their common tangents.
\label{thm:intersection}
\end{A}

\section* {Bibliography}
[AZ] {A. V. Akopyan and A. A. Zaslavsky},
{\it Geometry of Conics},
{Amer. Math. Soc.},
{Providence, RI},
{2007}
[Ba]  Bankoff,L.,{\it Are the twin circles of Archimedes really twins?}, Mathematics Magazine, 47: 214–218.
[B] Bergsten, C.,{\it
Magic Circles in the arbelos}
TMME, vol7, n.2-3, 209-222, ISSN 1551-3440

[D-L] Danneels, E., van Lamoen, F. : Midcircles and the arbelos. Forum
Geometricorum 7 (2007), 53-65.

[G]  Gheorghe, L.G. \emph{Apollonius' problem: in pursuit of a natural solution}, Int.J.Geom.,{\bf 9} (2) {\bf(2020)}, 39--51.

[G-S-O]  Glaeser,G., Stachel, H., Odehnal, B., \emph{The universe of Conics}, Springer Specktrum, Springer-Verlang Berlin Heidelberg, 2016.

[L-W] van Lamoen, F. and Weisstein, E. W. : "Pappus Chain." From MathWorld - A
Wolfram Web Resource.
[ http://mathworld.wolfram.com]

[O] Oller-Marcen´,A.,{\it The f-belos}
{\it Forum Geometricorum}
Volume 13 (2013) 103–-111.

[O-W] H. Okumara and M. Watanabe, The twin circles of Archimedes in a skewed arbelos, Forum
Geom., 4 (2004) 229–251.
[O-W] H. Okumura and M. Watanabe, A generalization of Power’s Archimedean circle, Forum Geom.,
6 (2006), 103–105.

[P] Poncelet, J.V.,   \emph {
Traité de propriétés
projectives des figures}, Gauthier-Villars, Paris,
1866

[S] Salmon, G.,\emph{ A treatise on conic sections}, Longman, Green, Reader and Dyer, London, 1869.

[S] Schoch, T.: “arbelos: The Woo circles”.
[http://www.retas.de/thomas

/arbelos/woo.html ]

[We] Weisstein, E. W.: "arbelos." From MathWorld - A Wolfram Web Resource.
[ http://mathworld.wolfram.com/arbelos.html ]

[W]  Woo, P., {\it Simple Constructions of the Incircle of an arbelos},
Forum Geometricorum
Volume 1 (2001) 133–-136.

\bigskip
 \bigskip
 \bigskip

DEPARTAMENTO DE MATEMÁTICA

UNIVERSIDADE FEDERAL DE PERNAMBUCO

RECIFE, (PE) BRASIL

\textit{E-mail address}:

\texttt{liliana@dmat.ufpe.br}
\bigskip
\bigskip
\end{document}